\documentclass[12pt]{amsart}

\usepackage{times,mathptmx}
\usepackage[utf8x]{inputenc}
\usepackage[OT2, T1]{fontenc}
\usepackage{amsmath,amsfonts,amssymb,mathrsfs,latexsym,amscd,amsthm}
\usepackage{dsfont,url}
\usepackage{verbatim}
\usepackage{hyperref}
\usepackage{enumitem}
\usepackage[all]{xy}
\theoremstyle{plain}

\newtheorem{theorem}{Theorem}[section]
\newtheorem{conjecture}[theorem]{Conjecture}
\newtheorem{lemma}[theorem]{Lemma}
\newtheorem{corollary}[theorem]{Corollary}
\newtheorem{proposition}[theorem]{Proposition}

\theoremstyle{definition}

\newtheorem{remark}[theorem]{Remark}

\newcommand{\BR}{\mathbf{R}}

\newcommand{\BZ}{\mathbf{Z}}
\newcommand{\BQ}{\mathbf{Q}}
\newcommand{\BN}{\mathbf{N}}
\newcommand{\BP}{\mathbf{P}}
\newcommand{\CS}{\mathcal{S}}
\newcommand{\BE}{\mathbf{E}}
\newcommand{\CE}{\mathcal{E}}
\newcommand{\CH}{\mathcal{H}}

\newcommand{\rank}{\operatorname{rank}}

\newcommand{\hetale}{\operatorname{H}^1_{\text{\'et}}}

\newcommand{\Spec}{\operatorname{Spec}}
\newcommand{\Sel}{\operatorname{Sel}^2}

\newcommand{\Kum}{\operatorname{Kum}}
\newcommand{\wKum}{\widetilde{\operatorname{Kum}}}
\newcommand{\tor}{\operatorname{tor}}
\newcommand{\kmodksq}{k^*/k^{*2}}

\DeclareSymbolFont{cyrletters}{OT2}{wncyr}{m}{n}
\DeclareMathSymbol{\Sha}{\mathalpha}{cyrletters}{"58}

\usepackage{geometry}
\title[Rational points on K3 surfaces]{Rational points on elliptic K3 surfaces \\ of quadratic twist type}
\author{Zhizhong Huang}
\address{Institut für Algebra, Zahlentheorie und Diskrete Mathematik\\ Leibniz Universität Hannover\\
	30167 Hannover, Deutschland.}
\email{zhizhong.huang@yahoo.com} 

\begin{document}
	\begin{abstract}
		In studying rational points on elliptic K3 surfaces of the form
		$$f(t)y^2=g(x),$$
		where $f,g$ are cubic or quartic polynomials (without repeated roots), we introduce a condition on the quadratic twists of two elliptic curves having simultaneously positive Mordell-Weil rank. We prove a necessary and sufficient condition for the Zariski density of rational points by using this condition, and we relate it to the Hilbert property. Applying to surfaces of Cassels-Schinzel type, we prove unconditionally that rational points are dense both in Zariski topology and in real topology.
	\end{abstract}
	\maketitle
	\tableofcontents
	
	\section{Introduction}
	\subsection{Topology of rational points on elliptic K3 surfaces -- an overview}
		A basic approach of studying how many rational points an algebraic variety possesses is via their density in Zariski topology and in analytic topology. A conjecture of Mazur states the following:
	\begin{conjecture}[\cite{Mazur}, Conjecture 1]\label{conj:Mazur} 
		Let $X$ be a smooth variety over $\BQ$. Suppose that $X(\BQ)$ is Zariski dense. Then the topological closure of $X(\BQ)$ in $X(\BR)$ is open (that is, a finite union of connected components of $X(\BR)$). 	\footnote{However this conjecture turns out to be false in its full generality, as demonstrated by counter-examples constructed by Colliot-Thélène, Skorobogatov and Swinnerton-Dyer in \cite[\S5]{CT-Sk-SD}. 
		Several amendments to Conjecture \ref{conj:Mazur} have also been proposed, cf. e.g. \cite[Conjectures 2 \& 3]{Mazur2} and \cite[Conjectures 4 \& 5]{CT-Sk-SD}.}
	\end{conjecture}
	
	A finer notion than the Zariski density is the so-called \emph{Hilbert property} (abbreviated as (HP) in the sequel, cf. \cite[Definition 3.1.2]{Serre}).
	Recall that for $X$ a variety over a number field $k$, the set $X(k)$ is \emph{thin} if there exist a finite number of dominant generically finite morphisms $\pi_i:Y_i\to X,1\leqslant i\leqslant r$ over $k$ without rational sections such that $X(k)\setminus \cup_{i=1}^r \pi_i(Y_i(k))$ is contained in a proper Zariski closed subset of $X$ (cf. \cite[\S3.1]{Serre}, \cite[\S 1.1]{Corvaja-Zannier}). We say that $X$ verifies (HP) over $k$ if $X(k)$ is not thin.
	The following conjecture was raised by Corvaja and Zannier:
	\begin{conjecture}[\cite{Corvaja-Zannier}, Question-Conjecture 1]\label{conj:CZ}
		Suppose that $X$ is smooth and algebraically simply connected. If $X(k)$ is Zariski dense, then $X$ verifies (HP) over the number field $k$.
	\end{conjecture}
	
	Throughout this article, by a K3 surface, we mean a projective, smooth, geometrically integral surface $\CE$ defined over a field $k$ having trivial canonical class and $\operatorname{H}^1(\CE,\mathcal{O}_\CE)=0$. Moreover it is called \emph{elliptic} if it admits over $k$ a fibration $\pi:\CE\to B$ onto a curve $B$ whose generic fibre is a smooth genus one curve.
	Results on producing rational points (over the ground field) on K3 surfaces defined over number fields are rare in current literature, although a result of Bogomolov-Tschinkel \cite{B-T} states that rational points are potentially dense on elliptic K3 surfaces (i.e. rational points defined over a finite extension of the ground field are Zariski dense). 
	

	In the context of an elliptic K3 surface $\CE\to\BP^1_{\BQ}$, Conjecture \ref{conj:Mazur} implies\footnote{To the best of author's knowledge, no counter-example to Conjecture \ref{conj:Mazur} is currently known for K3 surfaces.} (cf. \cite[p. 40]{Mazur}):
	\begin{conjecture}[\cite{Mazur}, Conjecture 4]
		For $t\in \BP^1_\BQ$, let $\CE_t$ be the fibre $\pi^{-1}(t)$. Then the family $\{\CE_t\}_{t\in \BP^1_\BQ}$ verifies one of the following exclusive conditions.
		\begin{enumerate}
			\item The elliptic curves $\CE_t$ has Mordell-Weil rank $0$ for all but a finite number of elements $t\in\BP^1(\BQ)$.
			\item The set $\{t\in\BP^1(\BQ):\rank(\CE_t(\BQ))>0\}$ is dense in $\BP^1(\BR)$.
		\end{enumerate}
	\end{conjecture}
	
	All the aforementioned conjectures for K3 surfaces are consequences of a stronger one due to Skorobogatov predicting that they satisfy weak approximation with Brauer-Manin obstruction (cf. \cite[p. 817]{H-Sk}).



	\subsection{Summary of results and layout of the paper}
	
	The object of study in this paper is a family of K3 surfaces of Cassels-Schinzel type \cite{C-S} with affine model
	\begin{equation}\label{eq:CS}
	\mathfrak{S}^{d,1}:d(1+T^4)Y^2=X^3-X~\subset~\mathbb{A}_{X,Y,T}^3,
	\end{equation}
	where $d\in\BN_{\geqslant 1}$. They arise as quadratic twist type elliptic pencils over $\BQ$ (parametrized by the coordinate $T$) that are \emph{isotrivial} (i.e. the $j$-invariant of this family is constant on $T$).
	Such elliptic surfaces admit a section over $\BQ$ at $X=\infty$ but the Mordell-Weil rank of the group formed by all sections over $\BQ$ (sometimes called ``geometric rank'') is zero. One main result of this paper is the following (cf. Section \ref{se:CSsurfaces}):
	\begin{theorem}\label{thm:CS}
		There exists an infinite set of square-free integers $d$ (including $1,7,\cdots$) such that the surface $\mathfrak{S}^{d,1}$ verifies Conjecture \ref{conj:Mazur}, and that $\mathfrak{S}^{d,1}(\BQ)$ is Zariski-dense.
	\end{theorem}
	Unless otherwise specified, all results in our article are unconditional on standard conjectures on elliptic curves. Theorem \ref{thm:CS} fits into the empirical fact that if there exist extra rational points outside the evident smooth rational curves on a K3 surface, they tend to be Zariski dense (cf. \cite{Swinnerton-Dyer}). See also a recent result of Gvirtz \cite[\S4.1]{Gvirtz} for another quadratic twist family of $j$-invariant $0$.
	
	Amongst classical approaches to the Zariski density of rational points on isotrivial elliptic surfaces, the method of \emph{variation of the root number} is pioneered by Rohrlich \cite{Rohrlich} and stimulates later works such as \cite{VA} and \cite{Desjardins}. As the family \eqref{eq:CS} has constant root number $-1$ (cf. $\S$\ref{se:rootnumber}), the Zariski density of rational points on such surfaces was previously conditional on the parity conjecture, asserting that the root number agrees with the parity of the Mordell-Weil rank of an elliptic curve, which is a weak version of the Birch-Swinnerton--Dyer conjecture. Our result also covers several elliptic surfaces in \eqref{eq:CS} with constant root number $+1$, for which there were no conditional results about the Zariski density of rational points previously.
	
	The Cassels-Schinzel surfaces \eqref{eq:CS} are particular types of (twisted) Kummer surfaces associated to the product of two (twisted) elliptic curves (cf. Section \ref{se:Kummer}). Let us fix a number field $k$, and $f,g$ two one-variable polynomials without repeated roots and such that $\deg f,\deg g\in \{3,4\}$. We define two hyperelliptic curves $$E_1:y^2=g(x),\quad E_2:s^2=f(t)$$ of genus one.
	The associated Kummer variety, denoted by $\Kum(E_1\times E_2)$, has $k$-affine model \begin{equation}\label{eq:Kum}
	f(T)Y^2=g(X)
	\end{equation}(cf. Section \ref{se:Kummervar}). To state our second main result which releases our strategy of proving Theorem \ref{thm:CS} and addresses Conjecture \ref{conj:CZ}, we introduce the condition ``\eqref{eq:bothrankpos}''\footnote{abbreviation for ``simultaneously positive rank''} (with parameter $[C]\in\kmodksq$) for the Kummer variety $\Kum(E_1\times E_2)$ as follows. 
	Let $E_1^C$ (resp. $E_2^C$) denote the quadratic twist of $E_1$ (resp. $E_2$) by $[C]\in\kmodksq$. We define the condition
	\begin{equation}\label{eq:bothrankpos}
	\rank(E_1^C(k))\rank(E_2^C(k))>0. \tag{SPR(C)}
	\end{equation}
	Here by writing ``rank'' we implicitly mean that $E_i(k)\neq\varnothing$ for any $i=1,2$, so that they are equipped with some group law.	   
	The condition \eqref{eq:bothrankpos} arises in a natural way from our approach of parametrizing rational points on \eqref{eq:Kum} (cf. Theorem \ref{thm:Z2torsors}) in a similar manner to \cite{CT-Sk-SD} (except that here we are taking unramified double coverings over an open (non-proper) variety).
	\begin{theorem}\label{thm:HP}
		\hfill
		\begin{enumerate}[label=(\roman*)]
			\item The set $\Kum(E_1\times E_2)(k)$ is Zariski dense if and only if there exists $[C]\in\kmodksq$ such that \eqref{eq:bothrankpos} holds.
			\item  If $\Kum(E_1\times E_2)$ verifies (HP) over $k$, then there exist infinitely many classes $[C]\in\kmodksq$ such that \eqref{eq:bothrankpos} holds.
		\end{enumerate}
	\end{theorem}
	\begin{remark}\label{rmk:oneCgivesinfinite}
		Suppose that Conjecture \ref{conj:CZ} holds for $\Kum(E_1\times E_2)$. We infer from Theorem \ref{thm:HP} that the existence of one class $[C]\in\kmodksq$ with \eqref{eq:bothrankpos} is equivalent to the existence of infinitely many such $[C]$. See also the discussion in \S\ref{se:simulnonvan}.
	\end{remark}
	We shall prove Theorem \ref{thm:HP} in $\S$\ref{se:proofSPR}, followed by different approaches to the condition \eqref{eq:bothrankpos} in $\S$\ref{se:simulnonvan}.
	Applying to the Cassels-Schinzel type surfaces \eqref{eq:CS}, we point out that the validity of (HP), and hence that of \eqref{eq:bothrankpos} with infinitely many $[C]$ for the surface \eqref{eq:CS} has striking consequence on the \emph{congruent numbers} (cf. Remark \ref{rmk:congruent}). Based on current unconditional results towards the detection of congruent numbers, we establish \eqref{eq:bothrankpos} with infinitely many $[C]\in\BQ^*/\BQ^{*2}$ for another family of surfaces similar to \eqref{eq:CS} with a different coefficient for the term $T^4$ (cf. Theorem \ref{th:mainthm1}), therefore providing evidence in favour of (HP).
	
	It is worth mentioning that, when $(\deg f,\deg g)=(3,3)$ (so that the surface \eqref{eq:Kum} is untwisted) and under some extra assumption on the $j$-invariants of two elliptic curves defined by $f$ and $g$, the Zariski density is known unconditionally by the work of Kuwata-Wang \cite{K-W}, where the third elliptic fibration (in coordinate $Y$) comes into play in their method. More recently, Demeio \cite{Demeio} surprisingly shows that, the Zariski density is equivalent to (HP), whence one deduces the equivalence between \eqref{eq:bothrankpos} (with a single $[C]$ or infinitely many by Remark \ref{rmk:oneCgivesinfinite}) and (HP) in this case (cf. also Corollary \ref{thm:degree3Kummer}). There again, the use of the elliptic fibration in $Y$ is essential. It does not seem evident to extend their results to cases where $\deg f$ or $\deg g$ is $4$ since the fibration in $Y$ is not of genus one any more.
	We give a short discussion in Section \ref{se:remarks}.
	

\subsection*{Acknowledgements.} We thank Marc Hindry for drawing our attention to the Cassels-Schinzel surfaces as well as enlightening guidance. We are grateful to Mike Bennett, Julian Demeio, Julie Desjardins, Damián Gvirtz, Emmanuel Lecouturier, Marusia Rebolledo, Alexei Skorobogatov for fruitful discussions, and to Yang Cao for generous help.
We gratefully thank the hospitality of Max-Planck-Institut für Mathematik in Bonn and the organizers of the CMO-BIRS conference 18w5012 in Oaxaca, where part of this work was done and reported. The numerous suggestions and corrections of anonymous referees have greatly improved the exposition.
While working on this project the author was partially supported by the project ANR Gardio and a Riemann Fellowship.

\section{Quadratic twist type K3 surfaces}\label{se:Kummer}
Twisted Kummer surfaces \eqref{eq:Kum} are examples of Kummer varieties (in the sense of \cite{H-Sk}). We start by recalling some basic facts on Kummer varieties in \S\ref{se:Kummervar}. For more details, see \cite[\S6]{H-Sk}. In \S\ref{se:Weierstrass} we recall a classical covering map between a hyperelliptic quartic and the elliptic curve defined its cubic resolvent. In \S\ref{se:eqKum}, we turn to giving affine models and quotient maps defining $\Kum(\mathcal{A})$ when $\mathcal{A}$ is the product of two hyperelliptic curves. In \S\ref{se:parametrization} we give explicit parametrization of rational points on $\Kum(\mathcal{A})$ by quadratic twists of $\mathcal{A}$. Throughout this section, we fix $k$ a field of characteristic $0$ and $\bar{k}$ an algebraic closure. 
\subsection{Preliminaries on Kummer varieties}\label{se:Kummervar}

Let $A$ be an abelian variety over $k$. Let $\mathcal{T}$ be a $k$-torsor under $A[2]$, and $\mathcal{A}$ be a $k$-torsor under $A$ whose class is the image of $[\mathcal{T}]\in \operatorname{H}^1(k,A[2])$ in $\operatorname{H}^1(k,A)$. Then $[\mathcal{T}]\in\operatorname{H}^1(k,A)[2]$ and $\mathcal{A}$ is equipped with a \emph{$2$-covering map} $\Phi:\mathcal{A}\to A$  (cf. \cite[Proposition 3.3.2]{Skorobogatov}) such that $\mathcal{T}=\Phi^{-1}(0)$. Indeed, we have $\mathcal{A}\simeq (A\times \mathcal{T})/A[2]$, where $A[2]$ acts diagonally on $A\times \mathcal{T}$ and $\Phi$ is induced by the first projection to $A$. The antipodal involution $\iota_A:A\to A$ (i.e. the multiplication by $-1$ morphism) commutes with the action of $A[2]$ and induces an involution $\iota_\mathcal{A}:\mathcal{A}\to \mathcal{A}$.
Let $$\pi_\mathcal{A}:\mathcal{A}\to \wKum(\mathcal{A}):=\mathcal{A}/\langle\iota_\mathcal{A}\rangle$$ be the quotient map. Note that the singular locus of the variety $\wKum(\mathcal{A})$ comprises precisely the fixed points of $\iota_\mathcal{A}$. The \emph{Kummer variety} $\Kum(\mathcal{A})$ associated to $\mathcal{A}$ is the minimal desingulization of $\wKum(\mathcal{A})$. 

To $C\in k^*$ we associate a $k$-torsor 
\begin{equation}\label{eq:Ctor}
\mathcal{R}_C=\Spec (k[T]/(CT^2-1))
\end{equation} 
under $\BZ/2$. By \emph{quadratic twist} of $A$ (resp. $\mathcal{A}$) by $C$, denoted by $A^C$ (resp. $\mathcal{A}^C$), we mean the diagonal quotient of $A\times_k \mathcal{R}_C$ (resp. $\mathcal{A}\times_k \mathcal{R}_C$) by $\BZ/2$, where $\BZ/2$ acts on $A$ (resp. $\mathcal{A}$) as $\iota_A$ (resp. $\iota_\mathcal{A}$). The projection $\mathcal{A}\times_k \mathcal{R}_C\to\mathcal{A}$ induces a morphism $\mathcal{A}^C\to \wKum(\mathcal{A})$. The antipodal involution $\iota_\mathcal{A}$ on $\mathcal{A}$ induces an involution $\iota_{\mathcal{A}^C}$ on $\mathcal{A}^C$, so that the quotient $$\pi_{\mathcal{A}^C}:\mathcal{A}^C\to \widetilde{\operatorname{Kum}}(\mathcal{A}^C):=\mathcal{A}^C/\langle\iota_{\mathcal{A}^C}\rangle$$ identifies $\wKum(\mathcal{A}^C)$ with $\wKum(\mathcal{A})$.
Since $A^C[2]\simeq_k A[2]$ canonically, we identify $\mathcal{A}^C$ as the torsor under $A^C$ defined by $[\mathcal{T}]\in \operatorname{H}^1(k,A[2])\simeq \operatorname{H}^1(k,A^C[2])$, equipped with a $2$-covering map $\Phi_C:\mathcal{A}^C\to A^C$. 

Since the action of the involution $\iota_{A^C}$ commutes with that of $A^C[2]$ on $A^C$, the $2$-covering map $\Phi_C$ commutes with the involutions on $A^C$ and $\mathcal{A}^C$: $$\Phi_C\circ\iota_{\mathcal{A}^C}=\iota_{A^C}\circ \Phi_C.$$
So $\Phi_C$ induces a covering map $\widetilde{\Phi}:\wKum(\mathcal{A})\to \wKum(A)$ between (singular) Kummer varieties.

\subsection{Hyperelliptic quartics}\label{se:Weierstrass}
Recall that a (smooth) hyperelliptic quartic has a smooth model as the intersection of two quadrics in $\BP^3$, whose affine plane model can be written as 
\begin{equation}\label{eq:hyperelliptic}
\mathcal{H}:Y^2=G(X)=aX^4+cX^2+dX+e,\quad a,c,d,e\in k,a\neq 0,
\end{equation}
such that $G$ has no multiple roots. The hyperelliptic involution $\iota_{\mathcal{H}}$ is given by $(X,Y)\mapsto(X,-Y)$.
Its geometric genus is $1$. 
It is classically known that the curve \eqref{eq:hyperelliptic} has the following $(\bar{k}/k)$-Weierstrass form (cf. \cite[Proposition 3.3.6]{Skorobogatov})
\begin{equation}\label{eq:Weierstrassmod}
\mathbb{E}:u^2=v^3-27Iv-27J,
\end{equation}
where
$$I=12ae+c^2,\quad J=72ace-27ad^2-2c^3,$$
are the cubic invariants of the polynomial $G$, and $\CH$ is a $k$-torsor under $\mathbb{E}$, whose class is the image of that of the zero-dimensional scheme $\mathcal{T}_{G}=\Spec(k[X]/(G(X)))$, as a torsor under the finite group scheme $\mathbb{E}[2]$. In particular, $\CH$ is equipped with a $2$-covering map $\Phi:\CH\to\mathbb{E}$. See \cite[Theorem 1 (ii)]{Connell} for explicit formulas.
If $G(X)=0$ has a root over $k$, then $\mathcal{T}_G\simeq_k\mathbb{E}[2]$, and by choosing this root as the origin, the isomorphism $\CH\simeq_k \mathbb{E}$ of elliptic curves arsing from the projection $(\mathbb{E}\times_k \mathcal{T}_G)/\mathbb{E}[2]\to\mathbb{E}$ commutes with the involutions on both sides
\footnote{If $\mathcal{H}(k)\neq \varnothing$, we have $\mathcal{H}\simeq_k \mathbb{E}$  by choosing some point as the origin. In general if this point is not a root of $G$, then the involution $\iota_\mathcal{H}$ does not commute with the group law.}.
For any $C\in k^*$, the quadratic twist $\mathcal{H}^C$ is defined by the affine equation $CY^2=G(X)$. As before, $[\CH^C]\in \operatorname{H}^1(k,\mathbb{E}^C)$ is induced by the $k$-torsor $\mathcal{T}_G$ under $\mathbb{E}^C[2]\simeq \mathbb{E}[2]$. In this way $\mathcal{H}^C$ is equipped with a $2$-covering map $\Phi_C:\mathcal{H}^C\to \mathbb{E}^C$ over $k$.



\subsection{Affine models of Kummer varieties associated to products of two hyperelliptic curves}\label{se:eqKum}
We fix two separable polynomials $f,g$. We assume that $\deg f,\deg g\in \{3,4\}$.
Define two hyperelliptic curves 
\begin{equation}\label{eq:twohyperellip}
\begin{split}
&E_1:y^2=g(x)\subset \mathbf{A}^2_{x,y};\\
&E_2:s^2=f(t)\subset \mathbf{A}^2_{s,t}.
\end{split}
\end{equation}
The quadratic twists of $E_1$ and $E_2$ by $C\in k^*$ have affine models
$$E_1^C:Cy^2=g(x),\quad E_2^C:Cs^2=f(t).$$
Their respective hyperelliptic involutions are
\begin{equation}\label{eq:ellipinvo}
\iota_{E_1^C}:(x,y)\mapsto (x,-y),\quad \iota_{E_2^C}:(t,s)\mapsto (t,-s).
\end{equation}

We now choose the affine model of $\Kum(E_1\times E_2)$ and the map $\pi_{E_1^C\times E_2^C}$ as follows. Let $\CS$ be defined by \eqref{eq:Kum}: \begin{equation}\label{eq:CSfamily}
\CS:f(T)Y^2=g(X)~\subset ~\mathbf{A}^3_{X,Y,T}
\end{equation}
Then outside the fixed point locus $(y=0) \times (s=0)$ of $\iota_{E_1^C}\times \iota_{E_2^C}$, the map
\begin{align}
\phi_C:E_1^C\times_k E_2^C&\dashrightarrow \mathcal{S},\label{eq:doublephi}\\
(x,y)\times(t,s)&\longmapsto (X,Y,T)=(x,\frac{y}{s},t)
\end{align}
is a (generically) double covering from the abelian surface $E_1^C\times E_2^C$ to $\mathcal{S}$.

When any one of the polynomials, say $f$, has degree $4$, we can suppose that $f=0$ has no root over $k$. Otherwise by the discussion in \S\ref{se:Weierstrass}, consider the elliptic curve $E_2^\prime$ defined by the cubic resolvent \eqref{eq:Weierstrassmod} of $f$ (which is of degree $3$). We have $E_2\simeq_k E_2^\prime$ commuting with involutions $\iota_{E_2}$ and $\iota_{E_2^\prime}$ and therefore $\Kum(E_1\times E_2)\simeq_k \Kum(E_1\times E_2^\prime)$.


\subsection{Parametrization of rational points}\label{se:parametrization}
With the affine model above, we prove the following folklore result, saying that rational points on (a Zariski open subset of) $\Kum(E_1\times E_2)$ are parametrized by those on a family of twisted abelian varieties.

The surface $\Kum(E_1\times E_2)$ has two genus one fibrations $\psi_X,\psi_T$ to $\BP^1$, that is, projections to $X$ and $T$ coordinates from the affine model $\CS$ \eqref{eq:CSfamily} in $\S$\ref{se:eqKum}. They are respectively induced by the projections $E_1\times E_2\to E_1 ,E_1\times E_2\to E_2$. Let $F_1$ (resp. $F_2$) be the image of $E_1[2]$ (resp. $E_2[2]$) in $\BP^1$. Then $\psi_X^{-1}(F_1)$ (resp. $\psi_T^{-1}(F_2)$) are four disjoint rational curves which are (geometrically) sections of $\psi_X$ (resp. $\psi_T$). 
Let 
 $$V=\Kum(E_1\times E_2)\setminus (\psi_X^{-1}(F_1) \cup\psi_T^{-1}(F_2)).$$ We identify $V$ with the open set $\CS\setminus(Y=0,\infty)$.
For each $[C]\in k^*/k^{*2}$, consider the open subset of $E_1^C\times E_2^C$:
\begin{equation}\label{eq:Uc}
U_C=(E_1^C\times E_2^C)\setminus \left((E_1^C[2]\times E_2^C)\cup (E_1^C\times E_2^C[2])\right).
\end{equation} 
\begin{theorem}\label{thm:Z2torsors}
We have:
	\begin{enumerate}[label=(\roman*)]
		\item  $(U_C)_{[C]\in\kmodksq}$ is a family of $\BZ/2$-torsors over $V$;
		\item  	$$V(k)=\bigsqcup_{[C]\in\kmodksq} \pi_{E_1^C\times E_2^C}(U_C(k)).$$
	\end{enumerate}

\end{theorem}
\begin{proof}
	The restricted morphism $\pi_{E_1^C\times E_2^C}|_{U_C}:U_C\to V$ is clearly flat and étale.
	Moreover, the divisor $\psi_T^{-1}(F_2)$ is a \emph{double fibre} on $V$ (with respect to the fibration $\psi_T$) (cf. \cite[\S2 p. 117]{CT-Sk-SD}). The function defining $\psi_T^{-1}(F_2)$ (an affine equation is given by $f$) is invertible on $V$. Applying \cite[Proposition 1.1 \& Theorem 2.1]{CT-Sk-SD} shows (i). The decomposition of rational points in (ii) now follows from the evaluation map
\begin{equation*}
	V(k)\longrightarrow \hetale(k,\BZ/2)=k^*/k^{*2}.
\end{equation*}
Indeed, for any $[C]\in\kmodksq$, the Kummer exact sequence twisted by $[C]$ gives
$$\xymatrix{1\ar[r]&\mathcal{R}_C\ar[r] &\mathbb{G}_{\operatorname{m}}\ar[r]^{\xi_C} &\mathbb{G}_{\operatorname{m}}\ar[r]& 1},$$
where $\mathcal{R}_C$ is defined by \eqref{eq:Ctor} and $\xi_C:\mathbb{G}_{\operatorname{m}}\to \mathbb{G}_{\operatorname{m}}$ is $t\mapsto Ct^2$.
The $\BZ/2$-torsor with class $[U_1]-[C]\in\hetale(V,\BZ/2)$ is isomorphic to $V\times_{\mathbb{G}_{\operatorname{m}},\xi_C} \mathbb{G}_{\operatorname{m}}$ where $V\to \mathbb{G}_{\operatorname{m}}$ given by the invertible function $f$. An affine model is then defined by (an open subset of) $$f(t)=Cs^2,\quad f(t)w^2=g(x).$$
And it is clearly isomorphic to $U_C$, under the change of variables $y=sw$.
\end{proof}
\begin{remark}\label{rmk:fibres}
	The proof above shows that, if we view the surface \eqref{eq:CSfamily} as a family of quadratic twists of $E_1$ by values of the polynomial $f(T)$, the fibres of $\psi_T$ which are isomorphic to $E_1^C$ for certain $[C]\in\kmodksq$ are parametrized by rational points on the curve $f(t)=Cs^2$, namely $E_2^C$.
\end{remark}

\section{The condition \eqref{eq:bothrankpos}}\label{se:ICproperty}


In this section we fix $k$ a number field, and $f,g$ two polynomials without repeated roots satisfying $3\leqslant \deg f,\deg g\leqslant 4$ and defining two hyperelliptic curves $E_1,E_2$ \eqref{eq:twohyperellip}.

The goal of this section is firstly to show Theorem \ref{thm:HP} (cf. $\S$\ref{se:proofSPR}). Secondly, we discuss various geometric and analytic methods towards the condition \eqref{eq:bothrankpos} (cf. \S\ref{se:simulnonvan}).


\subsection{Proof of Theorem \ref{thm:HP}}\label{se:proofSPR}

We begin with a basic observation towards the condition \eqref{eq:bothrankpos}. 
\begin{lemma}\label{le:Zdense}
	For any $[C]\in\kmodksq$, the condition \eqref{eq:bothrankpos} holds if and only if $(E_1^C\times E_2^C)(k)$ is Zariski dense.
\end{lemma}
\begin{proof}
	This follows from the general fact that for an elliptic surface $\pi:\CE\to B$, where $B$ is $\BP^1$ or an elliptic curve, $\CE(k)$ is Zariski dense if and only if infinitely many fibres have positive rank. A proof goes as follows. 	Take $D\subset \CE$ to be any integral curve. Then $\pi|_D$ is either constant or surjective . In the former case $D$ is contained in a fibre of $\pi$. For the latter case, $D$ intersects transversally with the generic fibre of $\pi$, and therefore it intersects transversally with all but finitely many fibres of $\pi$. So if $\rank(\CE_t(k))>0$ for infinitely many $t\in B(k)$, then $\CE(k)$ cannot not be contained in any finite union of integral curves, i.e., $\CE(k)$ is Zariski dense. Conversely, a standard application of Merel's uniformity theorem \cite[Corollaire]{Merel} shows that every point of $\CE(k)$ which is a torsion point of $\CE_t(k)$ for certain $t\in B(k)$ lies in a proper Zariski closed subset of $\CE$. Therefore if all but finitely many fibres have only torsion points, $\CE(k)$ cannot be Zariski dense.
	 
	Using the fact above, the statement of the lemma follows from considering the projections $ E_1^C\times E_2^C\to E_1^C$ and $ E_1^C\times E_2^C\to E_2^C$. 
\end{proof}
The next finiteness result is a slight generalization of an observation due to Rohrlich \cite[\S8 Lemma, p. 147]{Rohrlich}, who proved the statement over $\BQ$.
\begin{lemma}\label{le:Rohrlich}
	Let $E$ be an elliptic curve over a number field $k$ with equation $v^2=h(u)$,
	where $h(u)\in k[u]$ is separable of degree $3$.
	We fix the equation for $E^C$ the quadratic twist by $C\in k^\times$ to be $Cv^2=h(u)$.
	Then 
	$$\#\{u\in k:\exists [C]\in \kmodksq,\exists v\in k,(u,v)\in E^C(k)_{\tor}\}<\infty.$$
\end{lemma}
\begin{proof}
	We make use of Merel's uniformity theorem \cite{Merel} to extend the statement to arbitrary number fields. To simplify the reading we give details. 
	
	Since $E[2]\simeq E^C[2]$ for any $[C]\in\kmodksq$, the $u$-coordinates appearing in any $2$-torsion point different from the one at infinity of $E^C$ for some $[C]$ satisfy $h(u)=0$. 
	
	We claim that, for any fixed integer $N\geqslant 3$, if $E[N](k)\neq\{0_E\}$, then for any $[C]\in\kmodksq, [C]\neq [1]$, one has $E^C[N](k)=\{0_E\}$. Indeed, by choosing $P\in E[N](k)$ together with $0_E$ as a basis of $E[N]$, the image of the mod $N$ Galois representation $\varrho_{E,N}:\operatorname{Gal}(\bar{k}/k)\to \operatorname{GL}_2(\BZ/N)$ attached to $E$ is \begin{equation}\label{eq:galoisrep}
	\begin{pmatrix}
	1 & *  \\
	0 & \chi_N  \\
	\end{pmatrix},
	\end{equation}
	where $\chi_N$ is the mod $N$ cyclotomic character. And that of $\varrho_{E^C,N}$ attached to $E^C$ is conjugate inside $\operatorname{GL}_2(\BZ/N)$ to
	\begin{equation}\label{eq:galoisrep2}
	\mu_C\cdot\begin{pmatrix}
	1 & *  \\
	0 & \chi_N  \\
	\end{pmatrix},
	\end{equation}
	 where $\mu_C:\operatorname{Gal}(\bar{k}/k)\to\{\pm 1\}$ is the quadratic character associated to the extension $k(\sqrt{C})/k$. We see that elements of the form \eqref{eq:galoisrep2} cannot be conjugate to \eqref{eq:galoisrep}, unless $[C]=[1]$. This proves our claim.
	
	Finally, by \cite[Corollaire]{Merel}, there exists a constant $B=B(k)$ such that for every elliptic curve $\mathfrak{E}$ defined over $k$, every point in $\mathfrak{E}(k)_{\operatorname{tor}}$ has order $<B$. We have therefore proved that there are at most finitely many classes $[C]\in k^*/k^{*2}$ such that $E^C$ has a non-zero torsion point over $k$ of order $>2$. Therefore the collection of all possible $u$-coordinates of points of order $>2$ in $E^C(k)_{\operatorname{tor}}$ for all $[C]\in k^*/k^{*2}$ is finite. This finishes the proof.
\end{proof}
We now deduce the following, which may be seen as a quantitative version of \cite[Lemma 3.2]{Demeio}.
Recall \eqref{eq:Uc} the open set $U_C\subset E_1^C\times E_2^C$.
\begin{lemma}\label{le:simultaneoustwist}
	The union $$\bigcup_{\substack{[C]\in\kmodksq\\\eqref{eq:bothrankpos} \text{ does not hold}}}\pi_{E_1^C\times E_2^C}(U_C(k))$$ 
	is contained in a proper Zariski closed subset of $\Kum(E_1\times E_2)$.
\end{lemma}
\begin{proof}
	We first suppose that $\deg f=\deg g=3$. Then for any $[C]\in\kmodksq$ such that \eqref{eq:bothrankpos} does not hold, at least one of the groups $E_1^C(k),E_2^C(k)$ contains only torsion points by Lemma \ref{le:Zdense}. According to the equation of affine model $\CS$ \eqref{eq:CSfamily} and the expression of the map $\phi_C$ \eqref{eq:doublephi}, on applying Lemma \ref{le:Rohrlich} to both curves $E_1,E_2$, we conclude that there exists a Zariski closed subset $Z$ of $\mathcal{S}$ consisting of points $(X,Y,T)\in\CS(k)$ with finitely many possible choices of $X$-coordinates or $T$-coordinates, so that 
	$$\phi_C(U_C(k))\subset Z\cup\left( \left((Y=0)\cup (Y=\infty)\right)\cap \CS\right),$$
	the right-hand-side being a proper closed subset of $\CS$.
	
	Next, without loss of generality, let us assume $\deg f=4,\deg g=3$. Let $E_2^\prime$ be the elliptic curve defined by the cubic resolvent \eqref{eq:Weierstrassmod} of $f$ (cf. $\S$\ref{se:Weierstrass}). 
	Fix $C\in k^*$. If $E_2^C(k)=\varnothing$ then $(E_1^C\times E_2^C)(k)=\varnothing$. So it suffices to consider the case where $E_2^C(k)\neq \varnothing$, 
	and that the condition \eqref{eq:bothrankpos} holds for $\Kum(E_1\times E_2)$ amounts to saying it holds for $\Kum(E_1\times E_2^\prime)$. Let us consider the commutative diagram (the maps $\pi_{E_1^C\times E_2^C}$, $\pi_{E_1^C\times (E_2^\prime)^C}$ and $\widetilde{\Phi}$ are constructed in \S\ref{se:Kummervar} and $\S$\ref{se:eqKum})
	\begin{equation}\label{eq:coveringkum}
		\xymatrix{&E_1^C\times E_2^C \ar@{->}[r]^{\operatorname{Id}_{E_1^C}\times \Phi_C} \ar@{->}[d]^{\pi_{E_1^C\times E_2^C}} &E_1^C\times (E_2^\prime)^C\ar@{->}[d]^{\pi_{E_1^C\times (E_2^\prime)^C}}\\
		&\wKum(E_1\times E_2)\ar@{->}[r]^{\widetilde{\Phi}} &\wKum(E_1\times E_2^\prime),}
	\end{equation}
	where $\Phi_C:E_2^C\to (E_2^\prime)^C$ is the $2$-covering map (cf. \S\ref{se:Weierstrass}).
	By the argument in the preceding paragraph, we know that, by Lemma \ref{le:Zdense},
	$$\bigcup_{\substack{[C]\in\kmodksq\\(E_1^C\times (E_2^\prime)^C)(k) \text{ is not Zariski dense}}}\pi_{E_1^C\times (E_2^\prime)^C}(U_C^\prime(k)),$$
	where $U_C^\prime=(E_1^C\times (E_2^\prime)^C)\setminus ((E_1^C[2]\times (E_2^\prime)^C) \cup (E_1^C\times (E_2^\prime)^C[2]))$, is contained in a proper closed subset $W^\prime\subset\wKum(E_1\times E_2^\prime)$. Therefore 
	$$\bigcup_{\substack{[C]\in\kmodksq\\(E_1^C\times E_2^C)(k) \text{ is not Zariski dense}}}\pi_{E_1^C\times E_2^C}(U_C(k))\subset W:=\widetilde{\Phi}^{-1}(W^\prime),$$ where $W\subset \Kum(E_1\times E_2)$ is also proper Zariski closed.
	
	Finally, the case where $\deg f=\deg g=4$ is reduced to the above one, by considering the $2$-covering map between $E_1$ and the elliptic curve $E_1^\prime$ defined by the cubic resolvent \eqref{eq:Weierstrassmod} of $g$ and the covering map between $\wKum(E_1\times E_2)$ and $\wKum(E_1^\prime\times E_2^\prime)$.
\end{proof}
\begin{remark}
	Our proof above shows that, on the open subset $U=\Kum(E_1\times E_2)\setminus W$, for any $P\in U(k)$ (which might be empty), the fibres above $\psi_X(P)$ and $\psi_T(P)$ containing $P$ both have positive Mordell-Weil rank.
\end{remark}
\begin{proof}[Proof of Theorem \ref{thm:HP}]
	It is clear that both Zariski density and (HP) do not depend on the choice of birational models. 
	By Theorem \ref{thm:Z2torsors}, outside a proper Zariski closed subset, rational points on $\Kum(E_1\times E_2)$ are parametrized by the family $(U_C(k))_{[C]\in k^*/k^{*2}}$. Now (i) follows immediately from Lemmas \ref{le:Zdense} and \ref{le:simultaneoustwist}.
	 If $\Kum(E_1\times E_2)$ verifies (HP) over $k$, then $\Kum(E_1\times E_2)(k)$ is not dominated by rational points on any finite collection of  twisted abelian surfaces $E_1^C\times E_2^C$. On applying again Lemma \ref{le:simultaneoustwist}, the number of $[C]\in\kmodksq$ such that the condition \eqref{eq:bothrankpos} holds is necessarily infinite.
\end{proof}

\subsection{Remarks on simultaneous non-vanishing}\label{se:simulnonvan}
Let $j_{E_1}$ (resp. $j_{E_2}$) be the $j$-invariant of $E_1$ (resp. $E_2$). The following result had been discovered by Kuwata-Wang \cite[Theorem 5]{K-W} over $\BQ$ by making use of the existence of a rational curve on $\Kum(E_1\times E_2)$. We present a different proof based on Theorem \ref{thm:HP} and work of Demeio \cite{Demeio}.
\begin{corollary}\label{thm:degree3Kummer}
	Assume that $\deg f=\deg g=3$ and that $(j_{E_1},j_{E_2})\neq(0,0)$ or $(1728,1728)$. Then \eqref{eq:bothrankpos} holds for $\Kum(E_1\times E_2)$ with infinitely many classes $[C]\in\kmodksq$.
\end{corollary}
\begin{proof}
	The assumption fits into Kuwata-Wang's result \cite[Theorem 1]{K-W}. We briefly explain here how their method (stated over $\BQ$) can deduce that the rational points over $k$ an arbitrary number field are Zariski dense on $\Kum(E_1\times E_2)$. Indeed, after the base change $\zeta:\BP_k^1\to\BP_k^1,Y\mapsto Y^6$, we obtain a non-torsion section (the notation $\sigma_1$ in  \cite[p. 116]{K-W}) with respect to the new elliptic fibration $\CS\times_{\zeta,\BP_k^1}\BP_k^1\to\BP_k^1$ in coordinate $Y$. Therefore by Silverman's specialization theorem \cite{Silverman2}, almost all fibres have positive Mordell-Weil rank, whence the Zariski density (even the density in real topology if $k=\BQ$) holds. By Demeio's theorem \cite[Proposition 4.4]{Demeio}, the surface $\Kum(E_1\times E_2)$ verifies (HP) over $k$. The assertion now follows from Theorem \ref{thm:HP} (2).
\end{proof}

It would be interesting to compare with analytic approaches the condition \eqref{eq:bothrankpos}. A recent result of Petrow \cite[Theorem 2.2]{Petrow} studies simultaneous non-vanishing of derivatives of $L$-functions at the central value attached to two modular forms. He shows that, conditionally on GRH and under some assumption on conductors and root numbers, there exist infinitely many square-free integers $C$ such that
$\rank(E_1^C(\BQ))\rank(E_2^C(\BQ))=1$. 
See also a result Munshi \cite{Munshi} concerning simultaneous non-vanishing of $L$-functions at the central value, which has consequence in the number of $[C]\in \BQ^*/\BQ^{*2}$ such that $E_1^C$ and $E_2^C$ have simultaneously rank zero.
\section{Cassels-Schinzel type K3 surfaces}\label{se:CSsurfaces}
From now on assume $k=\BQ$. We study in this section the isotrivial elliptic pencil over $\BQ$ of $j$-invariant $1728$ with affine model
\begin{equation}\label{eq:Cassels-Schinzel}
\mathfrak{S}^{d,a}:d(1+a^2T^4)Y^2=X^3-X,
\end{equation}
for which we refer to being of Cassels-Schinzel type.
Here $d,a\in\BZ_{\neq 0}$ are square-free parameters.
(The original one constructed in \cite{C-S} is given by $a=1,d=7$.)
Let $\psi_T:\mathfrak{S}^{d,a}\to\BP^1$ be the elliptic fibration in coordinate $T$, write $\mathfrak{S}^{d,a}_t$ for the fibre over $t\in\BP^1(\BQ)$ and consider the set $$\mathcal{F}^{d,a}=\{t\in\BP^1(\BQ):\rank(\mathfrak{S}^{d,a}_t(\BQ))>0\}.$$ Cassels and Schinzel proved that the all sections with respect to $\psi_T$ are torsion (cf. \cite[Theorem 1 \& Corollary]{C-S}. Their argument yields the same result for twists by the polynomial $(1+4T^4)$). So the N\'eron-Silverman specialization theorem \cite{Silverman2} does not provide useful information on the ranks of fibres. We shall focus on the validity of condition \eqref{eq:bothrankpos} and its consequence on the density of rational points. 
Our main theorems are concerned with two families of surfaces of Cassels-Schinzel type and include Theorem \ref{thm:CS} as a special case.

\begin{theorem}\label{th:mainthm1}
	\hfill
	\begin{enumerate}
		\item There exists an infinite set $\mathfrak{D}_1$ of square-free integers $d$ containing $1,7,41,\cdots$ such that for the Kummer variety
		\begin{equation}\label{eq:CS1}
		\mathfrak{S}^{d,1}:d(1+T^4)Y^2=X^3-X,
		\end{equation} the condition \eqref{eq:bothrankpos} holds with at least one $[C]\in\BQ^*/\BQ^{*2}$ 
		\item 	There exists an infinite set $\mathfrak{D}_2$ of square-free integers $d$ such that for the Kummer variety
		\begin{equation}\label{eq:CS2}
		\mathfrak{S}^{d,2}:d(1+4T^4)Y^2=X^3-X,
		\end{equation} the condition \eqref{eq:bothrankpos} holds with infinitely many $[C]\in\BQ^*/\BQ^{*2}$. 
		\item For any $d\in\mathfrak{D}_1$ (resp. $d\in\mathfrak{D}_2$) rational points on \eqref{eq:CS1} (resp. on \eqref{eq:CS2}) are dense in the real topology, and the set $\mathcal{F}^{d,1}$ (resp. $\mathcal{F}^{d,2}$) is dense in $\BP^1(\BR)$. 
	\end{enumerate}
\end{theorem}


\subsection{Root number of congruent number elliptic curves}\label{se:rootnumber}
The \emph{root number} of an elliptic curve (denoted by $\omega(\cdot)$) appears in the functional equation of its $L$-function (cf. \cite[\S C.16]{Silverman1}, where it is called ``sign of the functional equation'').
Although the connection between root numbers and rational points of elliptic curves remains mostly conjectural, we find helpful in giving a short discussion here with regard to the variation of root numbers in the families \eqref{eq:CS1} \eqref{eq:CS2}.
Let 
\begin{equation}\label{eq:congruntnumbercurve}
\mathbf{E}:y^2=x^3-x
\end{equation} 
be the congruent number elliptic curve.
As usual $\mathbf{E}^D$ denotes the quadratic twist of $\mathbf{E}$ by a square-free integer $D$. We say that $D>0$ is a \emph{congruent number} if $\rank(\mathbf{E}^D(\BQ))>0$. 
A classical computation carried out by Birch and Stephens \cite{Birch-Stephens} shows that 
\begin{equation}\label{eq:congruentnumbers}
\omega(\mathbf{E}^D)=\begin{cases}
1 & ~ D\equiv 1,2,3 \mod8;\\
-1 & ~ D\equiv 5,6,7 \mod8.
\end{cases}
\end{equation}
The parity conjecture (cf. \cite[p. 119]{Rohrlich}) predicts that all such integers $D$ satisfying $\omega(\mathbf{E}^D)=-1$ should be congruent numbers.

We proceed to compute the root numbers of the family \eqref{eq:CS1} (as in \cite[p.347]{C-S}). Write $T=\frac{l}{m}$, with $(l,m)\in\BZ,\gcd(l,m)=1$.  A change of variables yields the equation
\begin{equation*}\label{eq:CasselsSchinzel2}
d(l^4+m^4)Y^2=X^3-X
\end{equation*}
for \eqref{eq:CS1}. Since for any odd prime $p\mid l^4+m^4$, we have $p\equiv 1\mod 8$ (cf. Proposition \ref{prop:a=1}).
We thus obtain a complete description of root numbers for fibres of \eqref{eq:CS1}, depending on the parity and the class modulo $8$ of $d$. For instance, if $d\equiv 7\mod 8$ (resp. $d\equiv 1\mod 8$) then the surface \eqref{eq:CS1} has constant root number $-1$ (resp. $+1$). For the family \eqref{eq:CS2} however, fibres tend to have varying root numbers (cf. \cite[Proposition 6.5.1]{Cohen}).
\subsection{Torsors under congruent number elliptic curves}\label{se:torsorundercong}
The purpose of this section is to discuss rational points on the hyperelliptic quartics
\begin{equation}\label{eq:HaC}
\CH_a^C:Cs^2=1+a^2t^4~\subset~ \mathbb{A}^2_{s,t},
\end{equation} 
with $C\in\BN_{\geqslant 1}$ square-free and $a\in \BZ_{\neq 0}$.

The (affine) real locus $\CH_a^C(\BR)$ of \eqref{eq:HaC} has two symmetric branches (one on the upside of the $t$-axis and the other one below). We see that even though the curve \eqref{eq:HaC} is not real connected, the involution $(s,t)\mapsto(-s,t)$ exchanges rational points between the two real components of \eqref{eq:HaC}. 
Consequently, if $\CH_a^C(\BQ)$ is infinite, then it is dense in all real connected components.

By the discussion in \S\ref{se:Weierstrass}, we know that
$\mathcal{H}^C_a$ is a torsor under $\mathbf{E}^{2aC}$, the quadratic twist of $\mathbf{E}$ \eqref{eq:congruntnumbercurve}  by $2aC$.
Hence if $\mathcal{H}^C_a(\BQ)\neq\varnothing$, then $\mathcal{H}^C_a\simeq_\BQ \mathbf{E}^{2aC}$.
However it happens that even if $\mathcal{H}_a^C$ is everywhere locally soluble, it still has no $\BQ$-point. Indeed, consider the class $[\mathcal{H}^C_a]\in \operatorname{H}^1(\BQ,\mathbf{E}^{2aC})$. According to the exact sequence of $\mathbb{F}_2$-vector spaces (cf. \cite[X.4. Theorem 4.2]{Silverman1}) induced by multiplication by $[2]$ map,
\begin{equation}\label{eq:exactseq2}
0\to\mathbf{E}^{2aC}(\BQ)/2\mathbf{E}^{2aC}(\BQ)\to\Sel(\BQ,\mathbf{E}^{2aC})\to\Sha^1(\BQ,\mathbf{E}^{2aC})[2]\to 0,
\end{equation}
if $[\mathcal{H}^C_a]$ maps to a non-zero element in $\Sha^1(\BQ,\mathbf{E}^{2aC})[2]$, then $\mathcal{H}^C_a(\BQ)=\varnothing$.
It follows from a formula of Monsky (cf. \cite[Appendix]{H-B}) that for $D\in\BN_{\geqslant 1}$ square-free,
\begin{equation}\label{eq:upperselmer}
	\dim_{\mathbb{F}_2}\Sel(\BQ,\BE^{D})\leqslant 2\Omega(D)+2,
\end{equation}
where $\Omega(D)$ is the number of odd prime divisors of $D$.

Up to sign, rational points on $\mathcal{H}_a^C$ are in one-to-one correspondence with non-zero primitive integral solutions of the Fermat-type equation $$X^4+a^2Y^4=CZ^2$$ by putting $$t=\frac{Y}{X},\quad s=\frac{Z}{X^2}.$$ Having Cassels-Schinzel's surface \eqref{eq:CS1} \eqref{eq:CS2} in mind, in what follows we shall be interested exclusively in the cases where $a=1$ and $a=2$. 
\subsubsection{Case $a=1$}\label{se:a=1}

\begin{proposition}\label{prop:a=1}
	The curve $\mathcal{H}^C_1$ is everywhere locally soluble if and only if 
	\begin{center}
			$(*)$ all odd prime factors of $C$ are $\equiv 1\mod 8 $.
	\end{center}
\end{proposition}
\begin{proof}
	This follows from the local solubility criterion \cite[Proposition 6.5.2]{Cohen} for the equation $X^4+Y^4=CZ^2$.
\end{proof}
According to \eqref{eq:congruentnumbers}, under the hypothesis $(*)$, the root number of $\BE^{2C}$ is zero, which conjecturally implies that $\rank(\BE^{2C}(\BQ))$ is even. Although most of such curves (are believed to) have rank zero, the following supplementary condition $(**)$ guarantees the existence of infinitely many points on $\CH_1^C$. 
\begin{corollary}\label{co:CohenC}
	Let $C\in\BN_{\geqslant 1}$ be square-free verifying $(*)$ and
	$$(**)\quad 2\Omega(C)=\rank(\BE^{2C}(\BQ)).$$
	Then $\#\mathcal{H}^C_1(\BQ)=\infty$.
\end{corollary}
\begin{proof}
	Indeed, from the exact sequence \eqref{eq:exactseq2} and inequality \eqref{eq:upperselmer}, we have
	\begin{align*}
	\dim_{\mathbb{F}_2}(\mathbf{E}^{2aC}(\BQ)/2\mathbf{E}^{2aC}(\BQ))&=\rank(\BE^{2C}(\BQ))+\dim_{\mathbb{F}_2}(\BE^{2C}(\BQ)[2])\\
	&=\rank(\BE^{2C}(\BQ))+2\\ &\leqslant \dim_{\mathbb{F}_2}\Sel(\BQ,\BE^{2C})\leqslant 2\Omega(C)+2.
	\end{align*}
	Hence under the assumption $(**)$, the inequalities above are all equalities. So in particular $\Sha^1(\BQ,\mathbf{E}^{2C})[2]=0$.
	This implies that $\mathcal{H}_1^C\simeq_\BQ \BE^{2C}$.
\end{proof}
\begin{remark}\label{rmk:congruent}
	Theorem \ref{thm:HP} (2) shows that the validity of (HP) for the surface $\mathfrak{S}^{d,1}$ implies the existence of infinitely many square-free integers $C$ verifying $(*)$ and $(**)$. We refer to Cohen's table \cite[p. 395]{Cohen} for all such $C$'s less than $10000$. However, such numbers seem to be very sparse. To the best of the author's knowledge, it is even not currently known whether there exist infinitely many congruent numbers verifying $(*)$.
	Even though elliptic surfaces in the family \eqref{eq:CS1} can have constant root number $-1$, preceding results indicate that it seems difficult to produce rational points on them, as the hyperelliptic curves $\CH_1^C$ parametrizing fibres of \eqref{eq:CS1} have few chances to possess infinitely many rational points.
\end{remark}

\subsubsection{Case $a=2$}\label{se:a=2}
The curve $\CH^C_2$ is a $\BQ$-torsor under $\BE^{4C}\simeq_\BQ\BE^C$.
Let us begin by describing a sufficient criterion of local solubility.
\begin{proposition}\label{prop:a=2}
	Assume that $C\equiv 5\mod 8$, and that for any prime $p\mid C$, $-4$ is a fourth root in $\mathbb{F}_p$. Then $\CH_2^C$ is everywhere locally soluble.
\end{proposition}
\begin{proof}
	This follows from the sufficient local solubility condition \cite[Proposition 6.5.1 (2c)\&(3e)]{Cohen} for non-zero integral solutions of the equation $X^4+4Y^4-CZ^2=0$.
\end{proof}

The condition $C\equiv 5\mod 8$ implies that $\omega(\BE^C)=-1$ by \eqref{eq:congruentnumbers}. Standard conjectures predict that
$$\dim_{\mathbb{F}_2} \Sel(\BQ,\BE^C)\equiv \rank(\BE^C(\BQ))\equiv 1\mod 2.$$
The following unconditional result provides an infinite set of square-free integers $C$ such that $\CH_2^C$ has infinitely many rational points.
\begin{proposition}\label{prop:H2p}
	For every prime number $p\equiv 5\mod 8$, we have $\#\CH_2^p(\BQ)=\infty$.
\end{proposition}
\begin{proof}
	First we check local solubility. Write $\left(\frac{\cdot}{\cdot}\right)$ (resp. $\left(\frac{\cdot}{\cdot}\right)_4$) for the quadratic (resp. quartic) residue symbol. 
	First, note that for any odd prime $p$, $$\left(\frac{4}{p}\right)_4=1\Leftrightarrow \text{ either }\left(\frac{2}{p}\right)=1\text{ or }\left(\frac{-2}{p}\right)=1.$$ Now for any prime $p\equiv 5\mod 8$, we have
	$$\left(\frac{-1}{p}\right)=1,\quad \left(\frac{-1}{p}\right)_4=\left(\frac{2}{p}\right)=\left(\frac{-2}{p}\right)=-1.$$ We deduce that $$\left(\frac{4}{p}\right)_4=-1.$$
	Therefore $$\quad\left(\frac{-4}{p}\right)_4=1.$$
	So for any prime $p\equiv 5\mod 8$, $\CH_2^p$ verifies the hypotheses of Proposition \ref{prop:a=2}. 
	Now we make use of a classical result due to Heegner (cf. \cite[Corollary 5.15 (1)]{Monsky}) affirming that 
	$$\rank(\BE^p(\BQ))=1,\quad \dim_{\mathbb{F}_2} \Sel(\BQ,\BE^p)=3.$$
Going back to the exact sequence \eqref{eq:exactseq2} we get $\Sha^1(\BQ,\BE^p)[2]=0$.
	So $\CH_2^p\simeq_\BQ\BE^p$ and $\rank(\CH_2^p(\BQ))=\rank(\BE^p(\BQ))=1$. 
\end{proof}

\subsubsection{Remark on $2$-descent}
Following Cohen \cite[\S6.5]{Cohen}, one can give explicit criterion to determine whether the class $[\CH_a^C]\in \Sel(\BQ,\BE^{2aC})$ comes from some non-torsion point of $\BE^{2aC}(\BQ)$ using descent by $2$-isogeny.
Indeed, consider the elliptic curve 
$\mathsf{E}^{aC}:aCy^2=x^3+x$.
There exists a $2$-isogeny $\phi:\BE^{2aC}\to \mathsf{E}^{aC}$, so that we get an exact sequence
$$\xymatrix{0\ar[r]&\BZ/2  \ar[r] &\mathbf{E}^{2aC}\ar[r]^{\phi} &\mathsf{E}^{aC}\ar[r]&0}.$$
This gives rise to
$$\xymatrix{
	&\mathsf{E}^{aC}(\BQ) \ar[r]^-\alpha &\operatorname{H}^1(\BQ,\BZ/2)\ar[r] &\operatorname{H}^1(\BQ,\mathbf{E}^{2aC}) \ar[r]^{\phi^\prime} &\operatorname{H}^1(\BQ,\mathsf{E}^{aC})},$$
where $\alpha$ is a partial $2$-descent map and $\phi^\prime$ is induced by $\phi$.
There is an elementary way of transforming points on $\CH_a^C$ into $\mathsf{E}^{aC}$ (cf. \cite[Proposition 6.5.5]{Cohen}). We conclude that $\psi^\prime[\CH_a^C]=0$ in $\operatorname{H}^1(\BQ,\mathsf{E}^{aC})$. So whether $[\CH_a^C]\in \operatorname{H}^1(\BQ,\BE^{2aC})$ is zero or not is characterised by the image of the map $\alpha$. See \cite[\S 6.5.3]{Cohen} for a simple description of $\alpha$.

\subsection{Proof of Theorem \ref{th:mainthm1}}
 By Theorem \ref{thm:Z2torsors}, outside a proper Zariski closed subset (comprising $24$ rational curves), rational points on the surface $\mathfrak{S}^{d,a}$ \eqref{eq:Cassels-Schinzel} are parametrized by the family $((\CH_a^C\times \mathbf{E}^{Cd})(\BQ))_{C\text{ square-free}}$ (an open subset of each), where $\CH_a^C$ is defined by \eqref{eq:HaC} and $\mathbf{E}$ is given by \eqref{eq:congruntnumbercurve}.
\subsubsection{Proof of (1)}  
We show that the infinite family of surfaces $\mathfrak{S}^{d,1}$ with $d=17p$ with $p$ prime $\equiv 5$ or $7\mod 8$ verifies \eqref{eq:bothrankpos} for at least one square-free $C>0$. We first extract the value $C=17$ from Cohen's table \cite[p. 395]{Cohen} satisfying the hypotheses of Corollary \ref{co:CohenC}, so that $\#\CH_1^{17}(\BQ)=\infty$. On the other hand, with this choice of $d$, we have $\mathbf{E}^{17\times 17p}\simeq_\BQ \mathbf{E}^p$. Thanks to \cite[Corollary 5.15 (2)]{Monsky}, any prime number $\equiv 5,7\mod 8$ is a congruent number. Therefore the condition (SPR(17)) holds for the infinite family $(\mathfrak{S}^{17p,1},p\equiv 5,7\mod 8)$.

To derive similar result for the surfaces $\mathfrak{S}^{1,1},\mathfrak{S}^{7,1}$ and $\mathfrak{S}^{41,1}$ with constant root numbers, it suffices to take $C=113$ and $C=257$ from Cohen's table \cite[p. 395]{Cohen} so that $\#\CH_1^{257}(\BQ),\#\CH_1^{113}(\BQ)=\infty$ and a consultation of elliptic curve database shows that $\rank(\mathbf{E}^{257}(\BQ))=2$,
$\rank(\mathbf{E}^{7\times 17}(\BQ))=1$ and $\rank(\mathbf{E}^{41\times 113}(\BQ))=2$. So the condition (SPR(257)) (resp. (SPR(17)), resp. (SPR(113))) holds for $\mathfrak{S}^{1,1}$ (resp. $\mathfrak{S}^{7,1}$, resp. $\mathfrak{S}^{41,1}$).

We conclude that for any $d\in\mathfrak{D}_1$, where
$$\mathfrak{D}_1=\{1,7,41,17p, ~(p\equiv 5,7\text{ mod }8)\},$$
the condition \eqref{eq:bothrankpos} holds for at least one square-free $C$.
\qed


\subsubsection{Proof of (2)}
Now we consider the family of surfaces $\mathfrak{S}^{d,2}$.
By a result of Monsky \cite[Corollary 5.15 (2)]{Monsky}, we have $\rank(\BE^C(\BQ))>0$ for any square-free integer $C$ of the form $pq$ or $2pq$ where $p\equiv 5\mod 8,q\equiv 3\text{ or }7\mod 8$ are prime numbers.
Therefore, for any $d=q$ or $2q$ with $q$ prime $\equiv 3\text{ or }7\mod 8$, we have
$\rank(\BE^{pd}(\BQ))\rank(\CH_2^p(\BQ))>0$ for any prime $p\equiv 5\mod 8$, using Proposition \ref{prop:H2p}.
We have shown that, for the infinite family $(\mathfrak{S}^{d,2})_{d\in\mathfrak{D}_2}$ of surfaces with $$\mathfrak{D}_2=\{q,2q,~(q\equiv 3\text{ or }7 \text{ mod } 8)\},$$ 
the condition (SPR($p$)) holds for any prime $p\equiv 5\mod 8$.
\qed

\subsubsection{Proof of (3)}
For any $C>0$, if $\rank(\BE^C(\BQ))>0$, then $\BE^C(\BQ)$ is dense in $\BE^C(\BR)$, since $\mathbf{E}^C[2]\simeq (\BZ/2)^2$ is defined over $\BQ$. The discussion in the beginning of \S\ref{se:torsorundercong} shows that $\#\CH_a^C(\BQ)=\infty$ implies $\CH_a^C(\BQ)$ is dense in $\CH_a^C(\BR)$. Once the condition \eqref{eq:bothrankpos} holds for one square-free $C>0$, the set $(\BE^C\times\CH_a^C)(\BQ)$ is dense in real topology.
Since the map $\phi_C$ \eqref{eq:doublephi} is generically \'etale, this implies that $\mathfrak{S}^{d,a}(\BQ)$ is dense in $\mathfrak{S}^{d,a}(\BR)$.
Moreover, the fibres $(\mathfrak{S}^{d,a}_t,t\in\BP^1(\BQ))$ are parametrized by rational points on the family of hyperelliptic curves $(\CH_a^C)_{C\text{ square-free}}$.
Then $\#\CH_a^C(\BQ)=\infty$ shows that the image of rational points under the projection $\CH_a^C\to\BP^1$ (in coordinate $t$), as a subset of $\mathcal{F}^{d,a}$, is dense in $\BP^1(\BR)$, so is the set $\mathcal{F}^{d,a}$ itself. 
Thus the proof of Theorem \ref{th:mainthm1} is completed. \qed
\section{Remarks and questions}\label{se:remarks}
The property (HP) for a variety $X$ over $k$ amounts to saying that the complement of the image of rational points on any finite collection of covers (that is, dominant generically finite rational maps of degree greater than two without rational sections) in $X(k)$ is still Zariski dense. As K3 surfaces are simply connected, any cover has non-empty ramification locus. For elliptic K3 surfaces of quadratic twist type \eqref{eq:CSfamily}, there is one particular type of covers which merits consideration. That is, those whose ramification locus is within the $16$ exceptional lines as blow-ups of the $16$ fixed points of the antipodal involution. One can show that any such cover which does not factor through an abelian surface of the form $E_1^C\times E_2^C$ is a cover between Kummer varieties (cf. the map $\widetilde{\Phi}$ constructed in \S\ref{se:Kummervar}). We have already seen a typical example in \eqref{eq:coveringkum}. The essential use of the third elliptic fibration in Demeio's argument, which does not exist for surfaces of Cassels-Schinzel type, successfully tackles with such covers, since it turns out that their generic fibres with respect to this fibration are curves of genus greater than two. 
We conclude by asking whether the validity of \eqref{eq:bothrankpos} with infinitely many square-free integers $C$ \footnote{As suggested by Demeio (private communication), one may need that such $C$'s form a non-thin subset of $\BQ$.} suffices to guarantee (HP) for surfaces of Cassels-Schinzel type \eqref{eq:Cassels-Schinzel}, or in an equivalent manner, whether or not there exist a finite number of coverings of Kummer varieties such that the images of the induced isogenies between twisted abelian surfaces contain $(\BE^{Cd}\times\CH_a^C)(\BQ)$ for any such $C$.

\end{document}